\documentclass[preprint,authoryear,3p,fleqn]{elsarticle}

\usepackage{amssymb,amsthm,amsmath,color}

\newtheorem{thm}{Theorem}
\newtheorem{lem}{Lemma}

\newtheorem{definition}{Definition}

\newtheorem{rmk}{Remark}
\journal{Statistics and Probability Letter}

\begin{document}

\begin{frontmatter}



\title{On the asymptotic variance of reversible Markov chain without cycles}


\author[Wu]{Chi-Hao Wu}
\ead{r05246009@ntu.edu.tw}

\author[Chen]{Ting-Li Chen\corref{cor1}}
\ead{tlchen@stat.sinica.edu.tw}

\cortext[cor1]{Corresponding author}
\address[Wu]{Institute of Applied Mathematical Sciences,
National Taiwan University, Taipei 10617, Taiwan}
\address[Chen]{Institute of Statical Sciences,
Academia Sinica, Taipei 11529, Taiwan}

\begin{abstract}
   Markov chain Monte Carlo(MCMC) is a popular approach to sample from high dimensional distributions, and the asymptotic variance is a commonly used criterion to evaluate the performance. While most popular MCMC algorithms are reversible, there is a growing literature on the development and analyses of nonreversible MCMC. \cite{Ch1} showed that a reversible MCMC can be improved by adding an antisymmetric perturbation. They also raised a conjecture that it can not be improved if there is no cycle in the corresponding graph.  In this paper, we present a rigorous proof of this conjecture. The proof is based on the fact that the transition matrix with an acyclic structure will produce minimum commute time between vertices.
\end{abstract}


\end{frontmatter}

\newcommand{\inner}[2]{\langle #1, #2 \rangle _{\pi}}

\section{Introduction}
   Markov chain Monte Carlo(MCMC) is a popular way to sample from high dimensional distributions with a wide variety of applications. For example, \cite{Diaconis2} provided several interesting examples including applications in cryptography and physics. Among the various applications of MCMC, one of the main purpose is to approximate the expectation of a specific function.   \\

   Let $\mathcal{I}$ be a fixed index set of the finite state space, and $\pi$ be a fixed probability measure on $\mathcal{I}$. The expectation of a real-valued function $f$ with respect to $\pi$ is denoted by $\pi(f) = \sum_{i \in \mathcal{I}} f(i) \pi(i)$. When the state space is extremely large, it is usually difficult to compute the expectation directly. In this case, we can construct a Markov chain with invariant distribution $\pi$. Then the empirical distribution of the samples $X_0, X_1, \dotsc, X_n, \dotsc$ converges to $\pi$. Therefore we can approximate $\pi(f)$ with the time average $\frac{1}{n}\sum_{k=0}^{n-1}f(X_k)$.   \\

   When we use MCMC algorithm with the transition matrix $P$ to approximate the expectation of a function $f$, its asymptotic variance $\nu (f, P, \pi)$ 
      \[
         \nu(f, P, \pi) = \lim_{n \to \infty} \mathbb{E}_{\mu} \Big[ \frac{1}{n} \sum_{k=0}^{n-1} f(X_k) - \pi(f) \Big]^2,
      \]
   where $\mu$ is an initial distribution, is a natural way to evaluate its performance. 
   Note that the asymptotic variance is independent to the initial distribution \citep{I}. Many researchers have studied how to optimize the performance of an MCMC algorithm based on this criterion. For examples, \citet{Peskun} showed that the transition matrix with larger values in off-diagonal entries has a smaller asymptotic variance. \cite{F} derived the optimal reversible transition matrix under the worst-case scenario, while \cite{C} obtained the optimal transition matrix under the average case scenario. \citet{W1} proposed a global optimization technique in constructing the optimal transition matrix.   \\

   A necessary condition for the convergence of MCMC is that $\pi$ is invariant under the transition $P$. However, the invariance is usually difficult to check when the state space is extremely large. The reversible condition, $\pi_i p_{i,j} = \pi_j p_{j,i}$, is a sufficient condition to guarantee the invariance, and most of the popular MCMC algorithms such as the Metropolis-Hastings algorithm \citep{Hastings,Met} and the Gibbs sampler \citep{Gibbs} make use of reversibility. Though the reversible condition makes the construction of a transition matrix easier, researchers have shown that reversible chains may not be as efficient as the non-reversible ones \citep{Diaconis1,M,H1,H2}. Recently, there is an increasing number of works on non-reversible MCMC \citep{C,Ch1,B,B1,Poncet}.   \\

   Since reversible transition matrices are found to be less efficient, two interesting questions arise naturally. What kinds of reversible transition matrix can be improved? How can we improve these reversible transition matrices? An earlier result can be found in \cite{Peskun}, and \cite{Ch1} showed that the chain can be uniformly better by adding an antisymmetric perturbation when the corresponding graph has cycles. They also proposed a conjecture that if the graph corresponding to the reversible transition matrix does not have any cycle, there does not exist a Markov chain that is uniformly better.   \\

   In this paper, we give a proof to the conjecture. It starts with some preliminaries listed in section~2, and the main result is presented in section~3. Finally, the conclusion is summarized in section~4.


\section{Preliminaries}
   In this paper, we only consider irreducible Markov chain. We start with introducing some notations and terminologies that will be used throughout the text.   \\

   A reversible Markov chain with a transition matrix $P$ has a corresponding undirected graph $G = (V, E)$, where $V$ is its state space and $E$ is the set of its edges. For any pair of states $i$ and $j$, $p_{i,j} > 0$ if and only if $(i, j) \in E$.   \\

   A cycle in a graph $G = (V, E)$ is a path which starts and ends at the same vertex without any other repeated vertices along the path. When the graph $G = (V, E)$ is undirected, a cycle must contain at least $3$ different vertices. If a graph has no cycle, we say it is acyclic.   \\

   An undirected acyclic graph $G = (V, E)$ has a tree structure. We say a vertex $i$ is a leaf if there is only one edge $(i, j)$ in $E$ that directly connects $i$ with another vertex $j$.

   \begin{definition}
      For real-valued functions $f$ and $g$ on $\mathcal{I}$, their inner product with respect to a probability distribution $\pi$ is defined as
      \[
         \inner{f}{g} = \sum_{i \in \mathcal{I}} f(i) g(i) \pi(i)
      \]
   \end{definition}

   \begin{lem} \label{L1}
      For any real-valued function $f$
      \[
         \nu(f, P, \pi) = 2 \inner{\big[(I-P+\Pi)^{-1}-\Pi\big] \{f-\pi(f)\}}{f-\pi(f)} - \inner{f-\pi(f)}{f-\pi(f)}
      \]
      where $\Pi$ is a matrix whose rows all equal $\pi$
   \end{lem}

   \begin{proof}
      See Theorem~$4.8$ in \cite{I}.
   \end{proof}

   When we consider only real-valued function in $\mathcal{N} = \{ f: \pi(f) = 0 \}$, the formula will be simplified into $\nu(f, P, \pi) = 2 \inner{\big[(I-P+\Pi)^{-1}-\Pi\big] f}{f} - \inner{f}{f}$. $(I-P+\Pi)^{-1}-\Pi$ is the so called fundamental matrix, and we denote it as $Z$. Since $\inner{f}{f}$ is independent to the transition matrix, the minimization of the asymptotic variance is equivalent to the minimization of $\inner{Z f}{f}$.    \\

   The first hitting time and the first return time are defined as
   \[
      T_i = \inf \{ t \geq 0: X_t = i \}
   \]
   and
   \[
      T_i^{+} = \inf \{ t \geq 1: X_t = i \}.
   \]

   In the following, we list some useful formulae that will be used in our proof.

   \begin{lem} \label{FUND}
      \[
         z_{ij} = \pi_j \mathbb{E}_{\pi} T_j - \pi_j \mathbb{E}_i T_j.
      \]
   \end{lem}

   \begin{proof}
      See Lemma~$11$ and Lemma~$12$ in Chapter~$2$ of \cite{A}.
   \end{proof}

   \begin{lem} \label{L2}
      \[
         \mathbb{E}_i T_j + \mathbb{E}_j T_k - \mathbb{E}_i T_k = \mathbb{P}_i (T_k < T_j) (\mathbb{E}_j T_k + \mathbb{E}_k T_j).
      \]
   \end{lem}

   \begin{proof}
      See Corollary~$10$ in Chapter~$2$ of \cite{A}.
   \end{proof}

   \begin{lem} \label{L3}
      \[
         \mathbb{E}_i T_j + \mathbb{E}_j T_i = \big[ \pi_i \mathbb{P}_i (T_j < T_i^{+}) \big]^{-1}.
      \]
   \end{lem}

   \begin{proof}
      See Corollary~$8$ in Chapter~$2$ of \cite{A}.
   \end{proof}

   \begin{definition}
      A transition matrix $P'$ is uniformly better than another $P$, denoted as $P' \succ P$, if
      \begin{enumerate}[1.]
         \item For any real function $f$
            \[
               \nu (f, P', \pi) \leq \nu(f, P, \pi).
            \]
         \item There exists at least one $f$ such that
            \[
               \nu (f, P', \pi) < \nu(f, P, \pi).
            \]
      \end{enumerate}
   \end{definition}

\section{Main result}
 In this section, we will prove the conjecture proposed in \citep{Ch1}.  We start with a few lemmas necessary for the proof of our main result.

   \begin{lem} \label{L4}
      Let $f_{ij}$ in $\mathcal{N}$ be $(0, \dotsc, 0, -\pi_i^{-1}, 0, \dotsc, 0, \pi_j^{-1}, 0, \dotsc, 0)^{T}$, while $a$ and $b$ be real numbers, then
         \begin{multline*}
            \inner{Z \{a f_{ij} + b f_{jk}\}}{a f_{ij} + b f_{jk}} =
               a^2 \big[ \mathbb{E}_i T_j + \mathbb{E}_j T_i \big] +
               b^2 \big[ \mathbb{E}_j T_k + \mathbb{E}_k T_j \big] +   \\
               ab  \big[ \mathbb{P}_k (T_i < T_j)(\mathbb{E}_i T_j + \mathbb{E}_j T_i) + \mathbb{P}_i (T_k < T_j)(\mathbb{E}_j T_k + \mathbb{E}_k T_j) \big]
         \end{multline*}
      where $i$, $j$ and $k$ are three different indices.
   \end{lem}

   \begin{proof}
      Let $z_j$ be the $j$th column of $Z$.
      \begin{align*}
         \inner{Z f_{ij}}{f_{ij}} &= \inner{-\pi_{i}^{-1} z_i + \pi_{j}^{-1} z_j}{f_{ij}}   \\
                                  &= -\pi_{i}^{-1} (-z_{ii} + z_{ji}) + \pi_{j}^{-1} (-z_{ij} + z{jj})   \\
                                  &= \mathbb{E}_i T_j + \mathbb{E}_j T_i \
      \end{align*}
      where the last equality is by Lemma \ref{FUND}. Similarly, we can show that $\inner{Z f_{jk}}{f_{jk}}=\mathbb{E}_j T_k + \mathbb{E}_k T_j$.   For the cross term,
      \begin{align*}
         \inner{Z f_{ij}}{f_{jk}} &= \inner{-\pi_{i}^{-1} z_i + \pi_{j}^{-1} z_j}{f_{jk}}   \\
                                  &= -\pi_{i}^{-1} (z_{ji} - z_{ki}) + \pi_{j}^{-1} (z_{jj} - z_{kj})   \\
                                  &= \mathbb{E}_k T_j + \mathbb{E}_j T_i - \mathbb{E}_k T_i   \\
                                  &= \mathbb{P}_k (T_i < T_j) (\mathbb{E}_i T_j + \mathbb{E}_j T_i)
      \end{align*}
      where the last equality is by Lemma \ref{L2}. Similarly, we can show that $\inner{Z f_{jk}}{f_{ij}}=\mathbb{P}_i (T_k < T_j) (\mathbb{E}_j T_k + \mathbb{E}_k T_j)$. These lead to the statement of the lemma.
   \end{proof}

   \begin{lem} \label{L5}
      Let $P$ and $P'$ be two transition matrices such that $\nu(f, P', \pi) \leq \nu(f, P, \pi)$ for all real-valued function $f$, and let $T_i$ and $T_i'$ be their first hitting time respectively.
      If $\mathbb{E}_i T_j + \mathbb{E}_j T_i = \mathbb{E}_i T_j' + \mathbb{E}_j T_i'$ and $\mathbb{P}_k(T_i < T_j) = \mathbb{P}_i(T_k < T_j)=0$,
      then $p'_{k,i} = p'_{i,k}=0$.
   \end{lem}

   \begin{proof}
      Let $x=\mathbb{E}_j T_k + \mathbb{E}_k T_j$, $y=\mathbb{P}_k (T_i < T_j)(\mathbb{E}_i T_j + \mathbb{E}_j T_i) + \mathbb{P}_i (T_k < T_j)(\mathbb{E}_j T_k + \mathbb{E}_k T_j)$, $x'$ and $y'$ be the corresponding terms for $P'$. Define $f_{ij}$ as in Lemma \ref{L4}. Since $\nu(a f_{ij} + b f_{jk}, P', \pi) \leq \nu(a f_{ij} + b f_{jk}, P, \pi)$ for any real number $a$, $b$ and $\mathbb{E}_i T_j + \mathbb{E}_j T_i = \mathbb{E}_i T_j' + \mathbb{E}_j T_i'$, by Lemma~\ref{L4}, we have
      \[
         b^2 x' + ab y' \leq b^2 x + ab y \quad \Longrightarrow \quad ab (y' - y) \leq b^2 (x - x').
      \]
      Since $ab$ can go to $\pm \infty$, the inequality hold only if $y' - y=0$.   \\

      Since $\mathbb{P}_k(T_i < T_j) = \mathbb{P}_i(T_k < T_j)=0$, we have $y = 0$, which implies $y' = 0$. Since $$y'=\mathbb{P}_k (T_i' < T_j')(\mathbb{E}_i T_j' + \mathbb{E}_j T_i') + \mathbb{P}_i (T_k' < T_j')(\mathbb{E}_j T_k' + \mathbb{E}_k T_j'),$$ $\mathbb{P}_k(T_i' < T_j') = \mathbb{P}_i(T_k' < T_j')=0$, which further implies that $p_{ki}' = p_{ik}'=0$.
   \end{proof}

   \begin{thm} \label{T1}
      Let $P$ be a reversible transition matrix of which the diagonal has at most one non-zero entry, and the corresponding graph $G = (V, E)$ is acyclic. If another transition matrix $P'$ has the property that $\nu(f, P', \pi) \leq \nu(f, P, \pi)$ for all real-valued function $f$, then $P = P'$.
   \end{thm}

   \begin{proof}
      Pick one leaf $i$ from $G$ with $p_{i,i}=0$. Since it is a leaf, there is only one neighbor connected to $i$, say $j$. Since $\mathbb{P}_i (T_j < T_i^{+})=1$, $\mathbb{E}_i T_j + \mathbb{E}_j T_i$ is minimized by Lemma~\ref{L3}. With the assumption $\nu(f, P', \pi) \leq \nu(f, P, \pi)$, it leads to $\mathbb{E}_i T_j + \mathbb{E}_j T_i=\mathbb{E}_i T'_j + \mathbb{E}_j T'_i$. Since $i$ is a leaf,  $\mathbb{P}_k(T_i < T_j) = \mathbb{P}_i(T_k < T_j)=0$ for $k\ne i,j$. Then by Lemma~\ref{L5},  $p'_{i,k}=p'_{k,i}=0$. Furthermore, $\mathbb{E}_i T_j + \mathbb{E}_j T_i=\mathbb{E}_i T'_j + \mathbb{E}_j T'_i$ leads to $\mathbb{P}_i (T'_j < {T'}_i^{+})=\mathbb{P}_i (T_j < T_i^{+})=1$, which implies $p'_{i,i}=0$ and $p'_{i,j}=1$. That means, the transitions of $i$-the row and $i$-th column of $P'$ and $P$ are the same.   \\

      Now that the $i$-th row and $i$-th column of $P'$ are determined, we remove the vertex $i$ from $V$ and the edge $(i,j)$ from $E$. Denote this new graph as $G_1=(V_1,E_1)$. In the following, we iteratively determine each row and column of $P'$ and update the associated graph with the similar rule.   \\

      Assume that we have determined $t-1$ rows of $P'$ and have the new graph $G_{t-1}=(V_{t-1},E_{t-1})$.
      Similar to the argument for the first iteration, now we pick one leaf $i_t$ from $G_{t-1}$ with $p_{i_t,i_t}=0$, and suppose that $j_t$ is its only neighbor in $G_{t-1}$. For convenience, we first define $G_t=(V_t,E_t)$ where $V_t=V_{t-1}\setminus i_t$ and $E_t=E_{t-1}\setminus (i_t,j_t)$.   \\


      If a vertex $k$ is in $V \setminus V_{t-1}$, it is processed in some earlier iteration; therefore, $p'_{\ell,k} = p_{\ell,k}$ and $p'_{k,\ell} = p_{k,\ell}$ for any $\ell$. Given a vertex $k$ in $V \setminus V_{t-1}$ with the property $p_{i_t, k} > 0$, since $k$ was a leaf in some earlier iteration, those vertices that the chain starting from $k$ can reach without passing $i_t$ are also in $V \setminus V_{t-1}$. That means the connected component containing $k$ in $(V \setminus \{i\}, E \setminus \{(i, j) \in E: j \in V\})$ belong to $V \setminus V_{t-1}$. Since rows and columns in $P'$ corresponding to vertices in that connected component are the same as those in $P$, any path from $k$ to $j_t$ in $G'$, the associated graph of $P'$, has to pass through $i_t$; this means that $\mathbb{P}_k (T'_{j_t} < {T'}_{i_t}^{+})=0$. Therefore, $p_{i_t,k}' \mathbb{P}_k (T'_{j_t} < {T'}_{i_t}^{+})=0$ for any $k \in V \setminus V_{t-1}$. Then


       \begin{align*}
         \mathbb{P}_{i_t} (T'_{j_t} < {T'}_{i_t}^{+})
            &= \sum_{k \in V_t} p_{i_t,k}' \mathbb{P}_k (T'_{j_t} < {T'}_{i_t}^{+})  \\
            &\leq \sum_{k \in V_t} p'_{i_t,k} = 1 - \sum_{k \notin V_t} p_{i_t,k} \\
            &= 1 - \sum_{k \notin V_t} p_{i_t,k} = p_{i_t,j_t}.   \\
            &= \mathbb{P}_{i_t} (T_{j_t} < {T}_{i_t}^{+}).
      \end{align*}
      With the assumption that $\nu(f, P', \pi) \leq \nu(f, P, \pi)$, we have $\mathbb{P}_{i_t} (T_{j_t} < T_{i_t}^{+}) = \mathbb{P}_{i_t} ({T'}_{j_t} < {T'}_{i_t}^{+})$, which leads to $\mathbb{E}_{i_t} T_{j_t} + \mathbb{E}_{j_t} T_{i_t} = \mathbb{E}_{i_t} T_{j_t}' + \mathbb{E}_{j_t} T_{i_t}'$ by Lemma~\ref{L3}. Since $i_t$ is a leaf in $G_{t-1}$, $\mathbb{P}_k(T_i < T_j) = \mathbb{P}_i(T_k < T_j)=0$  for any $k \in V_t \setminus j_t$, which implies $p'_{i_t,k}=p_{i_t,k}$ and $p'_{k,i_t}=p_{k,i_t}$ by Lemma~\ref{L5}. Furthermore, $\mathbb{E}_{i_t} T_{j_t} + \mathbb{E}_{j_t} T_{i_t}=\mathbb{E}_{i_t} T'_{j_t} + \mathbb{E}_{j_t} T'_{i_t}$ leads to $\mathbb{P}_{i_t} (T'_{j_t} < {T'}_{i_t}^{+})=\mathbb{P}_{i_t} (T_{j_t} < T_{i_t}^{+})$, which implies $p'_{i_t,i_t}=0$, $p'_{i,j}=p_{i,j}$ and  $p'_{j_t,i_t}=p_{j_t,i_t}$.  That means, the transitions of $i_t$-the row and $i_t$-th column of $P'$ and $P$ are the same.

      By the iterative process above, we can fully determine $P'=P$.

   \end{proof}

   \begin{rmk}
      Theorem \ref{T1} implies there does not exist $P'$ such that $P' \succ P$, which is the conjecture proposed in \citep{Ch1}.
   \end{rmk}


\section{Conclusion}

   \cite{Ch1} showed that we can improve the performance of a reversible transition matrix by adding an antisymmetric perturbation if there exists a cycle in the corresponding graph. They also proposed a conjecture that there does not exist a uniformly better one if the graph is acyclic. In this paper, we present a rigorous proof to this conjecture. The idea is that the acyclic structure leads to the minimization of the commute time between any two vertices. Along with the fact that the commute time between two vertices is equivalent to the asymptotic variance with respect to a specific $f$, a uniformly better transition has to be with the same acyclic graph structure. This will further imply the transition matrix is exactly the same.

\section*{Reference}
\bibliographystyle{elsarticle-harv}
\bibliography{mcmc}

\end{document}